\documentclass[reqno]{amsart}
\usepackage{hyperref}
\usepackage[a4paper, left=3cm, right=3cm, top=4cm, bottom=4cm]{geometry}
\vspace{9mm}
\begin{document}

\title[\hfilneg \hfil Time-discretization method]{Time-discretization method for a multi-term time fractional differential equation with delay}
	
\author[A. Khatoon, A. Raheem \& A. Afreen \hfil \hfilneg]
{A. Khatoon$^{*}$, A. Raheem \& A. Afreen}

\address{A. Khatoon
	\newline Department  of Mathematics, Aligarh Muslim University,
	\newline Aligarh-202002, India.}
\email{areefakhatoon@gmail.com}

\address{A. Raheem 
	\newline Department  of Mathematics, Aligarh Muslim University,
	\newline Aligarh-202002, India.}
\email{araheem.iitk3239@gmail.com}
	
\address{A. Afreen 
\newline Department of Mathematics, Aligarh Muslim University,
\newline Aligarh-202002, India.}
\email{afreen.asma52@gmail.com}
	
\renewcommand{\thefootnote}{}
\footnote{$^*$ Corresponding author:	
\url{ A. Khatoon (areefakhatoon@gmail.com)}}
	
\begin{abstract}
This paper discusses a multi-term time-fractional delay differential equation in a real Hilbert space. An iterative scheme for a multi-term time-fractional differential equation is established using Rothe's method. The method of semi-discretization is extended to this kind of time fractional problem with delay in the case that the time delay parameter $\nu >0$ satisfies $\nu\leq T$, where $T$ denotes the final time. We apply the accretivity of the operator $A$ in an iterative scheme to establish the existence and regularity of strong solutions to the considered problem. Finally, an example is provided to demonstrate the abstract result.
{\tiny } 
\end{abstract}	
\maketitle \numberwithin{equation}{section}
\newtheorem{theorem}{Theorem}[section]
\newtheorem{lemma}[theorem]{Lemma}
\newtheorem{problem}[theorem]{Problem}
\newtheorem{proposition}[theorem]{Proposition}
\newtheorem{corollary}[theorem]{Corollary}
\newtheorem{remark}[theorem]{Remark}
\newtheorem{definition}[theorem]{Definition}
\newtheorem{example}[theorem]{Example}
\allowdisplaybreaks
\section{\textbf{Introduction}}
Leibniz initially raised the idea of fractional derivatives in his letter to L'Hospital \cite{GW1695} dated September 30, 1695, when he questioned the meaning of ``half-order derivative." Many well-known mathematicians were fascinated by Leibniz's question. Since the 19th century, fractional calculus theory has developed rapidly and was the beginning of several disciplines. Many applications exist in various fields, including signal and image processing, porous media, optimal control,  fractional filters, fractals, soft matter mechanics,  etc. The non-integer order model describes a more accurate model than the integer order model, which is the main reason that the applications of fractional calculus are becoming more popular. We refer to the papers \cite{r5, BB2020, BB2022, y1, y2} and the references cited therein for the basics of fractional calculus and its applications.	
	
The theory of the approximate solution of differential equations has received much attention from numerous researchers. There are various methods for finding approximate solutions to differential equations. The Cauchy-Maruyama approximation \cite{G1955}, Caratheodory approximation \cite{D1989, X1994}, Euler-Maruyama approximation \cite{X2006, XCG2007}, Picard approximation \cite{D2002, RAR2017}, Faedo-Galerkin method \cite{AD2016, AM2021}, and Rothe's method \cite{AAAD2022, AD2014, K2020} are a few well-known techniques that are utilized to obtain the approximate solutions of differential equations.

The method of semi-discretization, named Rothe's approach, was developed in 1930 by Rothe \cite{E1930}  to handle a second-order scalar parabolic initial value problem. Rothe's method is used to demonstrate the existence and uniqueness of solutions for differential equations. Several researchers used this method; see, for instance, \cite{AR2010, A2011,MA2019, K2021}. Diffusion problems are also studied using this technique \cite{AA2018, NA2003, AD2011}. Recently, Rothe's method also discussed in variational and hemivariational inequalities \cite{KM2016, SS2019}.

In 2011, Dubey \cite{A2011} studied the existence of a solution to the delay differential equation using Rothe's method. In the same year, Raheem and  Bahuguna \cite{AD2011} investigated the existence and uniqueness of a strong solution for a fractional integral diffusion equation. In 2019, Migórski and Zeng \cite{SZS2019} demonstrated the existence of a solution for multi-term time fractional integral diffusion equations using the semi-discretization method. 

As mentioned above, fruitful results have been made in the case of single (multi-term) Caputo fractional diffusion equations with and without delay. There are many cases where the fractional diffusion equations with delay include not only the Caputo derivative but also a multi-term fractional derivative. Researchers in the work \cite{CS2009} also noted that ecological issues can be solved using delay diffusion equations. A strong motivation for investigating such equations comes from physics. Fractional diffusion equations describe anomalous diffusion on fractals (physical objects of fractional dimension, like some amorphous semiconductors or strongly porous materials; see \cite{VN2001, RJ2000} and references therein). Thus, it is meaningful to consider a class of multi-term Caputo fractional delay differential equations. 

Motivated by the above mentioned works \cite{A2011, SZS2019,  AD2011}, and following the approach used in these papers, we study this paper that deals with the existence of a strong solution of the following delay differential equation in a real Hilbert space ${H}$:	
\begin{eqnarray} \label{1.1}
	\left \{ \begin{array}{lll} \dfrac{d\vartheta(t)}{dt}+\sum\limits_{q=1}^{k}a_q{{}^C{D}_t^{\alpha_q}}\vartheta(t)+{A}\vartheta(t)=f(t,\vartheta(t-\nu)),&&t \in [0,T],
		\\\vartheta(t)=\chi(t),\quad t\in [-\nu,0],
	\end{array}\right.
\end{eqnarray}
where $a_q\geq 0$ are constants, ${}^C{D}_t^{\alpha_q}$ represents the Caputo derivative of order $0<\alpha_q<1$ for $q=1,2,\dots,k,$~  $\nu>0,~ T <\infty,~-{A}$ generates a $C_0$ semigroup of contractions in ${H}$, $f:[0,T]\times C\big([-\nu,0];{H}\big)\to {H}$, $\chi\in C\big([-\nu,0];{H}\big)$. Here, $C([-\nu,t];{H}) $ for $t\in [0,T]$ is the set of all continuous functions from $[-\nu,t]$ into ${H}$ and the space $\mathcal{C}_t:= C([-\nu,t];{H}),~t\in [0,T]$ denotes the Banach space with norm
\begin{eqnarray*}
	\|\vartheta\|_t:=\sup\limits_{-\nu\leq s\leq t}\|\vartheta(s)\|, ~~ \vartheta \in \mathcal{C}_t ,
\end{eqnarray*}
where $\|\cdot\|$ represents the norm in ${H}.$

In \cite{KMA2022}, Bockstal et al. studied a damped variable order fractional subdiffusion equation with time delay in a finite-dimensional case, and they proved the existence of a weak solution by using a discretization approach. In this problem \cite{KMA2022}, the method of semi-discretization is extended to this kind of time fractional parabolic problem with delay in the case that the time delay parameter $\nu>0$ satisfies $\nu\leq T$, where $T$ denotes the final time. Our result is the generalization of this work but in the case of constant order with a single Caputo derivative term in abstract space.

In \cite{JCY2022}, Du et al. considered the following  multiterm Caputo–Katugampola fractional delay integral diffusion equations in Hilbert space $H$ 
\begin{eqnarray*} \label{ch1.1}
	\left \{ \begin{array}{lll} \dfrac{\partial\vartheta(t)}{\partial t}+{A}\vartheta(t)=\sum\limits_{i=1}^{k}a_i\Big({}_0I_t^{\alpha_i,\rho_i}\vartheta(t)\Big)+f(t,\vartheta_t),&&t \in [0,T],
		\\\vartheta_0=\chi\mbox{ on } [-\nu,0],
	\end{array}\right.
\end{eqnarray*}
in which they studied the existence of a strong solution by employing Rothe’s method. If we take $\rho_i=1$, then the above problem changes into the following form
\begin{eqnarray*}
	\left \{ \begin{array}{lll} \dfrac{\partial\vartheta(t)}{\partial t}+{A}\vartheta(t)=\sum\limits_{i=1}^{k}a_i\Big({I}_t^{\alpha_i}\vartheta(t)\Big)+f(t,\vartheta_t),&&t \in [0,T],\\
		\vartheta_0=\chi\mbox{ on } [-\nu,0],
	\end{array}\right.
\end{eqnarray*}
But in our paper, we considered the multi-term fractional delay differential equation in a real Hilbert space ${H}$ given by (\ref{1.1}), where the Caputo derivative is considered instead of fractional integral.

The outline of this paper is as follows. We provide some necessary definitions, assumptions, and lemma in Section 2. In Section 3, we first discretize the interval $[-\nu,0]$ using a uniform time mesh in delay problem \cite{KMA2022}. To be able to apply Rothe’s method, we need to restrict the time frame to $[0, T_0]$ with $T_0:= \lfloor \frac{T}{\nu}\rfloor\nu$ assuming $\nu\leq T$. Furthermore, priori estimates and a few necessary lemmas are proved in Section 3. The existence of a strong solution is presented in Section 4, while in Section 5, an example is provided in support of main result. A conclusion is included at the end for future work.
\section{\textbf{Preliminaries and Assumptions}}
\begin{definition}\cite{I1999}
	The fractional derivative of Caputo  type for a function $g$ of order $0<\alpha<1$ is defined as
	\begin{eqnarray*}
		{}^C{D}_t^{\alpha_q} g(t)=\frac{1}{\Gamma(1-\alpha)}\int_0^t\frac{g'(s)}{(t-s)^{\alpha}}ds.
	\end{eqnarray*}
\end{definition}
\begin{definition}\cite{I1999}
	The fractional integral of Riemann-Liouville type for a function $g$ of order $\alpha>0$ is defined as
	\begin{eqnarray*}
		{I}^{\alpha}_t g(t)=\frac{1}{\Gamma(\alpha)}\int_0^t\frac{g(s)}{(t-s)^{1-\alpha}}ds.
	\end{eqnarray*}
\end{definition}
\begin{definition}\cite{A1983}
	Let $(X, \|\cdot\|)$ be a real normed space and $(X^*, \|\cdot\|_*)$ be its dual. Then for each $x\in X,$ the duality mapping $J$ is defined as
	\begin{eqnarray*}
		J(x)=\left\{x^*\in X^*~ |~ \langle x,x^*\rangle =\|x\|^2=\|x^*\|_*^2\right\},
	\end{eqnarray*} 
	where $\langle x,x^*\rangle$ represents the value of $x^*$ at $x.$
\end{definition}
It is well-known, see \cite[Theorem 1.2]{V1993}, that if $X^*$ is strictly convex, then the duality
mapping $J$ is single valued and demicontinuous. In particular, if $X$ is a Hilbert
space, then the duality mapping $J$ becomes the identity operator $I$.
\begin{definition}
	Let $X$ be a Banach space. A single valued operator $A$ is called accretive if 
	\begin{eqnarray*}
		\left\langle {A}x_1-Ax_2, J\left(x_1-x_2\right)\right\rangle \geq 0 \quad \text{ for each } \quad x_1, x_2\in D(A).
	\end{eqnarray*}
	A single valued accretive operator $A$ is called $m$-accretive if 
	\begin{eqnarray*}
		R(I+\lambda A)=X \quad \text { for }\quad \lambda>0.	
	\end{eqnarray*}
\end{definition}
For an accretive operator $A$, we introduce the following sequence of operators $J_{\lambda}$ and $A_{\lambda}$ from $R(I+\lambda A)$ into $X$ by
\begin{eqnarray*}
	J_{\lambda}x&=& (I+\lambda A)^{-1}x \text{ for } x\in R(I+\lambda A), \\
	A_{\lambda}x&=&AJ_{\lambda}x= \lambda^{-1}(I-J_{\lambda})x \text{ for } x\in R(I+\lambda A),
\end{eqnarray*}
where $R(I+\lambda A)$ denotes the range of operator $I+\lambda A$ and the operator $A_{\lambda}$ is called the Yosida approximation of $A$ (for details, see \cite[p.151]{Reich1980} or \cite[p.101]{V1993}).
\begin{proposition}\label{accrativeproperty}\cite{V1993}
	Let $A: X\to 2^X$ be an $m$-accretive operator. Then $A$ is closed and if $\lambda_n\in \mathbb{R}$ and $x_n\in X$ are such that
	\begin{eqnarray*}
		\lambda_n\to 0, \quad x_n\to x, \quad A_{\lambda_n}x_n\to y, \quad \text{as}\quad  n\to \infty, 
	\end{eqnarray*}
	then $y\in Ax$. If $X^*$ is uniformly convex, then $A$ is demiclosed, and if
	\begin{eqnarray*}
		\lambda_n\to 0, \quad x_n\to x, \quad A_{\lambda_n}x_n\rightharpoonup y, \quad \text{as}\quad  n\to \infty, 
	\end{eqnarray*}
	then $y\in Ax$. 
\end{proposition}
Recall that an operator$A: X\to 2^X$ is said to be closed, if $x_n\to x, y_n\to y$ and $y_n\in Ax_n$, then $y\in Ax$. Also, A is said to be demiclosed, if  $x_n\to x, y_n\rightharpoonup y$ and $y_n\in Ax_n$ yield  $y\in Ax$. Here, the symbols ``$\rightarrow$'' and ``$\rightharpoonup$''
stand for the strong convergence in $X$ and weak convergence in $X$.

\begin{lemma}\label{2}\cite{A1983}
	If $-{A}$ is an infinitesimal generator of a $C_0$ semigroup of contractions, then ${A}$ is an $m$-accretive operator.
\end{lemma}
\begin{definition} 
	The state function $\vartheta\in C\big([-\nu,T];{H}\big)$ is said to be a strong solution of (\ref{1.1}), if it satisfies the following:
	\begin{itemize}
		\item [1.]$\vartheta\in D(A)$ a.e. on $[0, T]$ and $\vartheta(t)=\chi(t),~ t\in [-\nu, 0]$;
		
		\item[2.] $\vartheta$ is differentiable a.e. on $[0, T]$;
		
		\item[3.]  $\vartheta$ satisfies (\ref{1.1}) a.e. on $[0, T]$.
		
	\end{itemize}
\end{definition}
The following assumptions hold throughout the paper.
\begin{itemize}
	\item[(A1)] $-{A}$ is an infinitesimal generator of a $C_0$ semigroup of contractions.	
	
	\item [(A2)]
	The nonlinear function $f:[0,T]\times C\big([-\nu,0];{H}\big)\rightarrow {H}$ satisfies the following condition
	\begin{eqnarray*}
		\big\|f(t_1,x_1)-f(t_2, x_2)\big\|
		&\leq& L_{f}\left(\big|t_1-t_2\big|+\big\|x_1-x_2\big\|\right),	
	\end{eqnarray*}
	for all $t_1,t_2\in [0,T],~x_1,x_2\in B_{\epsilon}\big(C\big([-\nu,0];{H}\big),\chi(0)\big),~L_{f}$ is a positive constant, where
	\begin{eqnarray*}
		B_{\epsilon}\big(C\big([-\nu,0];{H}\big),\chi(0)\big)=\Big\{y \in C\big([-\nu,0];{H}\big): \big\|y-\chi(0)\big\|<\epsilon \Big\}.
	\end{eqnarray*}
	\item[(A3)] $\chi(t) \in D({A}),~ t \in [-\nu,0].$
	
	\item[(A4)] The function $\chi(t)$ satisfies the following condition
	\begin{eqnarray*}
		\big\|\chi(t_1)-\chi(t_2)\big\|
		&\leq& L_{\chi}|t_1-t_2|,	
	\end{eqnarray*}
	for all $t_1,t_2\in [-\nu,0]$, where $L_{\chi}$ is a positive constant.
\end{itemize} 
\section{\textbf{Discretization scheme and a priori estimates}}
Rothe’s method is utilized to show the existence of a solution. First, the time interval $[-\nu,0]$ is discretized by a
time step $h<\min\{1,\nu\}$ defined by $h=\frac{\nu}{n}$,  where $n$ is a positive integer. Next, we define
\begin{eqnarray*}
	T_0:= \left\lfloor \frac{T}{\nu}\right\rfloor\nu, \quad \nu\leq T.
\end{eqnarray*}
We will show the existence of a solution on the time interval $[0, T_0]$. The time discrete points are given by $t_j=jh$ for all $-n\leq j\leq m$, where $m= \frac{T_0}{h}= \left\lfloor \frac{T}{\nu}\right\rfloor n$. The $\vartheta_j$ denotes the approximate solution at time $t = t_j$ for  $-n\leq j\leq m$. Moreover,
the time derivative of $\vartheta_j$ at time $t = t_j$ is approximated by the backward Euler finite-difference formula
\begin{eqnarray*}
	\delta\vartheta_j=\dfrac{\vartheta_j-\vartheta_{j-1}}{h},\quad 1\leq j\leq m.
\end{eqnarray*}
Also, the Caputo derivative of $\vartheta$ at $t=t_j$ is approximated by (see \cite{AD2019})
\begin{eqnarray}\label{3.5}
	{}^C{D}_t^{\alpha_q}\vartheta_j= \frac{h^{1-\alpha_q}}{\Gamma(2-\alpha_q)}\sum\limits_{i=1}^{j}b_{j-i}\delta\vartheta_i,
\end{eqnarray}
where $b_i=(i+1)^{1-\alpha_q}-i^{1-\alpha_q}$ for $i=1,2,\dots.$

\noindent Then, the problem (\ref{1.1}) is approximated at time $t = t_j$ as follows:
\begin{problem}
	Find $\{\vartheta_j\}\subset {H}$ such that 
	\begin{eqnarray} \label{3.1}
		\dfrac{\vartheta_j-\vartheta_{j-1}}{h}+\sum\limits_{q=1}^{k}a_q{{}^C{D}_t^{\alpha_q}}\vartheta_j+{A}\vartheta_j=f\big(t_j,\vartheta_{j-n}\big),\quad 1\leq j \leq m,
	\end{eqnarray}
	where
	\begin{eqnarray*}
		\vartheta_j=\chi(t_j),  \quad -n\leq j\leq 0.
	\end{eqnarray*}
\end{problem}
Using equation (\ref{3.5}), the discrete problem can be equivalently written as
\begin{eqnarray} \label{3.6}
	\delta\vartheta_j+\sum\limits_{q=1}^{k}\frac{a_qh^{1-\alpha_q}}{\Gamma(2-\alpha_q)}\sum\limits_{i=1}^{j}b_{j-i}\delta\vartheta_i+{A}\vartheta_j=f\big(t_j,\vartheta_{j-n}\big).
\end{eqnarray}
The existence of unique $\vartheta_j\in D(A)$ satisfying (\ref{3.1}) is a consequence of the m-accretivity of $A$.
\begin{lemma}\label{3l1}
	If the assumptions (A1)–(A4) are satisfied, then for any $j = 1, 2,\dots,m$,
	\begin{eqnarray}\label{3.7}
		\max_{j=1,2,\dots,m}\left\|\vartheta_j-\chi(0)\right\|\leq C.	
	\end{eqnarray}
\end{lemma}
\begin{proof} From (\ref{3.6}), for $j = 1$, one has
	\begin{eqnarray*}
		\vartheta_1-\chi(0)+\sum\limits_{q=1}^{k}\frac{a_qh^{1-\alpha_q}}{\Gamma(2-\alpha_q)}\big(\vartheta_1-\chi(0)\big)+h{A}\vartheta_1=hf\big(t_1,{\vartheta}_{1-n}\big).
	\end{eqnarray*}
	After rearranging, we get
	\begin{eqnarray*}
		\bigg(1+\sum\limits_{q=1}^{k}\frac{a_qh^{1-\alpha_q}}{\Gamma(2-\alpha_q)}\bigg)\big(	\vartheta_1-\chi(0)\big)+h{A}\big(\vartheta_1-\vartheta_0\big)=-hA\vartheta_0+hf\big(t_1,{\vartheta}_{1-n}\big).
	\end{eqnarray*}
	Multiplying the two sides of the above equation by $\vartheta_1-\vartheta_0$ and using the m-accretivity of ${A},$ we have
	\begin{eqnarray*}
		\bigg(1+\sum\limits_{q=1}^{k}\frac{a_qh^{1-\alpha_q}}{\Gamma(2-\alpha_q)}\bigg)\big\langle\vartheta_1-\chi(0),\vartheta_1-\chi(0)\big\rangle+h\big\langle{A}(\vartheta_1-\vartheta_0),\vartheta_1-\vartheta_0\big\rangle\\\hspace{4cm}=h\big\langle-A\vartheta_0,\vartheta_1-\vartheta_0\big\rangle +h\big\langle f\big(t_1,{\vartheta}_{1-n}\big), \vartheta_1-\vartheta_0\big\rangle.
	\end{eqnarray*}
	which implies that
	\begin{eqnarray*}
		\bigg(1+\sum\limits_{q=1}^{k}\frac{a_qh^{1-\alpha_q}}{\Gamma(2-\alpha_q)}\bigg)\big\|	\vartheta_1-\chi(0)\big\|^2\leq h\big\|A\vartheta_0\big\|\big\|\vartheta_1-\vartheta_0\big\|+h\big\|f\big(t_1,{\vartheta}_{1-n}\big)\big\|\big\|\vartheta_1-\vartheta_0\big\|.
	\end{eqnarray*}
	We divide both sides of the above inequality by $\big\|\vartheta_1-\vartheta_0\big\|$, then we get
	\begin{eqnarray*}
		\bigg(1+\sum\limits_{q=1}^{k}\frac{a_qh^{1-\alpha_q}}{\Gamma(2-\alpha_q)}\bigg)\big\|	\vartheta_1-\chi(0)\big\|\leq h\big\|A\vartheta_0\big\|+h\big\|f\big(t_1,{\vartheta}_{1-n}\big)\big\|.
	\end{eqnarray*}
	Since $	\bigg(1+\sum\limits_{q=1}^{k}\dfrac{a_qh^{1-\alpha_q}}{\Gamma(2-\alpha_q)}\bigg)>1,$ we have
	\begin{eqnarray*}
		\big\|\vartheta_1-\chi(0)\big\|
		&\leq& h\big\|A\vartheta_0\big\|+h\big\|f(t_1,{\vartheta}_{1-n})\big\|\\
		&\leq&h\big\|A\vartheta_0\big\|+h\left[\big\|f(t_1,\chi(t_{1-n}))-f(0, \chi(t_{-n}))\big\|+\big\|f(0,\chi(t_{-n}))\big\| \right].
	\end{eqnarray*}
	Using (A2) and (A4), we have
	\begin{eqnarray*}
		\big\|\vartheta_1-\chi(0)\big\|
		&\leq&h\big\|A\vartheta_0\big\|+h\left[L_f\left[|t_1-0|+\|\chi(t_{1-n})-\chi(t_{-n})\|\right]+\big\|f(0,\chi(t_{-n}))\big\| \right]\\
		&\leq&h\big\|A\vartheta_0\big\|+h\left[L_f\left[|t_1|+L_{\chi}|t_{1-n}-t_{-n}|\right]+\big\|f(0,\chi(t_{-n}))\big\| \right]\\
		&\leq&C_1
	\end{eqnarray*}
	where $
	C_1=T\big\|A\vartheta_0\big\|+T\left[L_fT(1+L_{\chi})+\big\|f(0,\chi(t_{-n}))\big\| \right]$.\\
	
	\noindent For $j\geq2$ in (\ref{3.6}), we get
	\begin{eqnarray*}
		\vartheta_j-\vartheta_{j-1}+\sum\limits_{q=1}^{k}\frac{a_qh^{1-\alpha_q}}{\Gamma(2-\alpha_q)}\sum\limits_{i=1}^{j}b_{j-i}\big(\vartheta_i-\vartheta_{i-1}\big)+h{A}\vartheta_j=hf(t_j,{\vartheta}_{j-n}).
	\end{eqnarray*}
	We can write it as
	\begin{eqnarray*}
		\vartheta_j-\vartheta_{j-1}+\sum\limits_{q=1}^{k}\frac{a_qh^{1-\alpha_q}}{\Gamma(2-\alpha_q)}\left[\vartheta_j+\sum\limits_{i=1}^{j-1}(b_{j-i}-b_{j-i-1})\vartheta_i-b_{j-1}\vartheta_{0}\right]+h{A}\vartheta_j=hf(t_j,{\vartheta}_{j-n}).
	\end{eqnarray*}
	After some simplification, we write it as
	\begin{eqnarray*}
		&&\bigg(1+\sum\limits_{q=1}^{k}\dfrac{a_qh^{1-\alpha_q}}{\Gamma(2-\alpha_q)}\bigg)\vartheta_j+h{A}\vartheta_j\\
		&&\quad=\vartheta_{j-1}+\sum\limits_{q=1}^{k}\frac{a_qh^{1-\alpha_q}}{\Gamma(2-\alpha_q)}\sum\limits_{i=1}^{j-1}b_{j-i-1}\vartheta_i-\sum\limits_{q=1}^{k}\frac{a_qh^{1-\alpha_q}}{\Gamma(2-\alpha_q)}\sum\limits_{i=1}^{j-1}b_{j-i}\vartheta_i\\
		&&\quad\quad+\sum\limits_{q=1}^{k}\frac{a_qh^{1-\alpha_q}}{\Gamma(2-\alpha_q)}b_{j-1}\vartheta_0+hf(t_j,{\vartheta}_{j-n}).
	\end{eqnarray*}
	Since $\bigg(1+\sum\limits_{q=1}^{k}\dfrac{a_qh^{1-\alpha_q}}{\Gamma(2-\alpha_q)}\bigg)>1,$ we have
	\begin{eqnarray*}
		&&\vartheta_j-\chi(0)+h{A}(\vartheta_j-\chi(0))\\
		&&\quad=\vartheta_{j-1}-\chi(0)-hA\chi(0)+\sum\limits_{q=1}^{k}\frac{a_qh^{1-\alpha_q}}{\Gamma(2-\alpha_q)}\sum\limits_{i=1}^{j-1}b_{j-i-1}\vartheta_i-\sum\limits_{q=1}^{k}\frac{a_qh^{1-\alpha_q}}{\Gamma(2-\alpha_q)}\sum\limits_{i=1}^{j-1}b_{j-i}\vartheta_i\\
		&&\quad\quad+\sum\limits_{q=1}^{k}\frac{a_qh^{1-\alpha_q}}{\Gamma(2-\alpha_q)}b_{j-1}\vartheta_0+hf(t_j,{\vartheta}_{j-n}).
	\end{eqnarray*}
	Multiplying both sides of the inequality by $\vartheta_j-\chi(0)$, we have
	\begin{eqnarray*}
		&&\langle\vartheta_j-\chi(0), \vartheta_j-\chi(0) \rangle+\langle h{A}(\vartheta_j-\chi(0)),\vartheta_j-\chi(0) \rangle\\	&&\quad=\langle\vartheta_{j-1}-\chi(0), \vartheta_j-\chi(0) \rangle-\langle hA\chi(0), \vartheta_j-\chi(0)\rangle+\sum\limits_{q=1}^{k}\frac{a_qh^{1-\alpha_q}}{\Gamma(2-\alpha_q)}\sum\limits_{i=1}^{j-1}b_{j-i-1}\langle\vartheta_i, \vartheta_j-\chi(0)\rangle\\
		&&\quad-\sum\limits_{q=1}^{k}\frac{a_qh^{1-\alpha_q}}{\Gamma(2-\alpha_q)}\sum\limits_{i=1}^{j-1}b_{j-i}\langle\vartheta_i, \vartheta_j-\chi(0)\rangle+\sum\limits_{q=1}^{k}\frac{a_qh^{1-\alpha_q}}{\Gamma(2-\alpha_q)}b_{j-1}\langle\vartheta_0,\vartheta_j-\chi(0)\rangle\\
		&&\quad +\langle hf(t_j,{\vartheta}_{j-n}), \vartheta_j-\chi(0)\rangle.
	\end{eqnarray*}
	By the m-accretivity of operator ${A},$ we have
	\begin{eqnarray}\label{delta}
		\left\|\vartheta_j-\chi(0)\right\|	&\leq&\left\|\vartheta_{j-1}-\chi(0)\right\|+h\|A\chi(0)\|+\sum\limits_{q=1}^{k}\frac{a_qh^{1-\alpha_q}}{\Gamma(2-\alpha_q)}\sum\limits_{i=1}^{j-1}b_{j-i-1}\left\|\vartheta_i-\chi(0)\right\|\nonumber\\
		&&+\sum\limits_{q=1}^{k}\frac{a_qh^{1-\alpha_q}}{\Gamma(2-\alpha_q)}\sum\limits_{i=1}^{j-1}b_{j-i}\|\vartheta_i-\chi(0)\|+\sum\limits_{q=1}^{k}\frac{a_qh^{1-\alpha_q}}{\Gamma(2-\alpha_q)}\sum\limits_{i=1}^{j-1}b_{j-i-1}\left\|\chi(0)\right\|\nonumber\\
		&&+\sum\limits_{q=1}^{k}\frac{a_qh^{1-\alpha_q}}{\Gamma(2-\alpha_q)}\sum\limits_{i=1}^{j-1}b_{j-i}\|\chi(0)\|+\sum\limits_{q=1}^{k}\frac{a_qh^{1-\alpha_q}}{\Gamma(2-\alpha_q)}b_{j-1}\|\vartheta_0\|\nonumber\\
		&&+h\|f(t_j,{\vartheta}_{j-n})\|.
	\end{eqnarray}
	For the function $f$, we can write it as 
	\begin{eqnarray*}
		\|f(t_j,{\vartheta}_{j-n})\|
		&\leq& \sum\limits_{i=1}^{j}\left\|f(t_{i},{\vartheta}_{i-n})-f(t_{i-1},{\vartheta}_{i-1-n})\right\|+ \|f(t_{j-j},{\vartheta}_{j-j-n})\|\\
		&=&\sum\limits_{i=1}^{j}\left\|f(t_{i},\chi(t_{i-n}))-f(t_{i-1},\chi(t_{i-1-n}))\right\|+ \|f(0,\chi(t_{-n}))\|.
	\end{eqnarray*} 
	Using the assumption (A2) and (A4), we have
	\begin{eqnarray}\label{fdifference}
		\|f(t_j,{\vartheta}_{j-n})\|
		&\leq& \sum\limits_{i=1}^{j}L_f\left[|t_{i}-t_{i-1}|+\left\|\chi(t_{i-n})-\chi(t_{i-1-n})\right\|\right]+ \|f(0,\chi(t_{-n}))\|\nonumber\\
		&\leq& \sum\limits_{i=1}^{j}L_f\left[|t_{i}-t_{i-1}|+L_{\chi}|t_{i-n}-t_{i-1-n}|\right]+ \|f(0,\chi(t_{-n}))\|\nonumber\\
		&=&jhL_f(1+L_{\chi})+\|f(0,\chi(t_{-n}))\|.
	\end{eqnarray}
	Using (\ref{fdifference}) in (\ref{delta}), we have
	\begin{eqnarray}
		\left\|\vartheta_j-\chi(0)\right\|	&\leq&\left\|\vartheta_{j-1}-\chi(0)\right\|+h\|A\chi(0)\|+\sum\limits_{q=1}^{k}\frac{a_qh^{1-\alpha_q}}{\Gamma(2-\alpha_q)}\sum\limits_{i=1}^{j-1}b_{j-i-1}\left\|\vartheta_i-\chi(0)\right\|\nonumber\\
		&&+\sum\limits_{q=1}^{k}\frac{a_qh^{1-\alpha_q}}{\Gamma(2-\alpha_q)}\sum\limits_{i=1}^{j-1}b_{j-i}\|\vartheta_i-\chi(0)\|+\sum\limits_{q=1}^{k}\frac{a_qh^{1-\alpha_q}}{\Gamma(2-\alpha_q)}\sum\limits_{i=1}^{j-1}b_{j-i-1}\left\|\chi(0)\right\|\nonumber\\
		&&+\sum\limits_{q=1}^{k}\frac{a_qh^{1-\alpha_q}}{\Gamma(2-\alpha_q)}\sum\limits_{i=1}^{j-1}b_{j-i}\|\chi(0)\|+\sum\limits_{q=1}^{k}\frac{a_qh^{1-\alpha_q}}{\Gamma(2-\alpha_q)}b_{j-1}\|\vartheta_0\|\nonumber\\
		&&+h\left[jhL_f(1+L_{\chi})+\|f(0,\chi(t_{-n}))\|\right].
	\end{eqnarray} 
	Summing up the above inequality from $2$ to $l$, ~$2\leq l\leq m$, we have
	\begin{eqnarray*}
		\left\|\vartheta_l-\chi(0)\right\|	&\leq&\left\|\vartheta_{1}-\chi(0)\right\|+(l-1)h\|A\chi(0)\|+\sum\limits_{q=1}^{k}\frac{a_qh^{1-\alpha_q}}{\Gamma(2-\alpha_q)}\sum\limits_{j=2}^{l}\sum\limits_{i=1}^{j-1}b_{j-i-1}\left\|\vartheta_i-\chi(0)\right\|\\&&+\sum\limits_{q=1}^{k}\frac{a_qh^{1-\alpha_q}}{\Gamma(2-\alpha_q)}\sum\limits_{j=2}^{l}\sum\limits_{i=1}^{j-1}b_{j-i}\|\vartheta_i-\chi(0)\|+\sum\limits_{q=1}^{k}\frac{a_qh^{1-\alpha_q}}{\Gamma(2-\alpha_q)}\sum\limits_{j=2}^{l}\sum\limits_{i=1}^{j-1}b_{j-i-1}\left\|\chi(0)\right\|\\&&+\sum\limits_{q=1}^{k}\frac{a_qh^{1-\alpha_q}}{\Gamma(2-\alpha_q)}\sum\limits_{j=2}^{l}\sum\limits_{i=1}^{j-1}b_{j-i}\|\chi(0)\|+\sum\limits_{q=1}^{k}\frac{a_qh^{1-\alpha_q}}{\Gamma(2-\alpha_q)}\sum\limits_{j=2}^{l}b_{j-1}\|\vartheta_0\|\\
		&&+h\sum\limits_{j=2}^{l}\left[jhL_f(1+L_{\chi})+\|f(0,\chi(t_{-n}))\|\right].
	\end{eqnarray*}

	Combining this estimate with the fact that \cite{SZS2019}
	\begin{eqnarray*}
		\sum\limits_{j=2}^{l}\sum\limits_{i=1}^{j-1}b_{j-i-1}\|\vartheta_i-\chi(0)\|&=&\sum\limits_{j=2}^{l}\sum\limits_{i=1}^{j-1}\left[(j-i)^{1-\alpha_q}-(j-i-1)^{1-\alpha_q}\right]\|\vartheta_i-\chi(0)\|\\
		&=&\sum\limits_{j=2}^{l} \|\vartheta_i-\chi(0)\| \sum\limits_{i=1}^{l-j}\left[(l-i)^{1-\alpha_q}-(l-i-1)^{1-\alpha_q}\right]\\
		&\leq &l^{1-\alpha_q}\sum\limits_{j=2}^{l} \|\vartheta_i-\chi(0)\|,
	\end{eqnarray*}
	and the boundedness of $\vartheta_1-\chi(0)$, we have
	\begin{eqnarray*}
		\left\|\vartheta_l-\chi(0)\right\|	&\leq&\left\|\vartheta_{0}-\chi(0)\right\|+lh\|A\chi(0)\|+2\sum\limits_{q=1}^{k}\frac{a_qh^{1-\alpha_q}}{\Gamma(2-\alpha_q)}\left(l^{1-\alpha_q}\sum\limits_{j=2}^{l} \|\vartheta_i-\chi(0)\|\right)\\
		&&+\sum\limits_{q=1}^{k}\frac{a_qh^{1-\alpha_q}}{\Gamma(2-\alpha_q)}\sum\limits_{j=2}^{l}(j-1)^{1-\alpha_q}\left\|\chi(0)\right\|+\sum\limits_{q=1}^{k}\frac{a_qh^{1-\alpha_q}}{\Gamma(2-\alpha_q)}\sum\limits_{j=2}^{l}\left[j^{1-\alpha_q}-1\right]\|\chi(0)\|\\
		&&+\sum\limits_{q=1}^{k}\frac{a_qh^{1-\alpha_q}}{\Gamma(2-\alpha_q)}\sum\limits_{j=1}^{l}b_{j-1}\|\vartheta_0\|+h^2L_f(1+L_{\chi})\sum\limits_{j=1}^{l}j+ lh\|f(0,\chi(t_{-n}))\|\\
		&\leq&\left\|\vartheta_{0}-\chi(0)\right\|+lh\|A\chi(0)\|+2\sum\limits_{q=1}^{k}\frac{a_qT^{1-\alpha_q}}{\Gamma(2-\alpha_q)}\sum\limits_{j=2}^{l} \|\vartheta_i-\chi(0)\|\\
		&&+\sum\limits_{q=1}^{k}\frac{a_qh^{1-\alpha_q}}{\Gamma(2-\alpha_q)}\sum\limits_{j=2}^{l}(j-1)^{1-\alpha_q}\left\|\chi(0)\right\|+\sum\limits_{q=1}^{k}\frac{a_qh^{1-\alpha_q}}{\Gamma(2-\alpha_q)}\sum\limits_{j=2}^{l}\left[j^{1-\alpha_q}-1\right]\|\chi(0)\|\\
		&&+\sum\limits_{q=1}^{k}\frac{a_qh^{1-\alpha_q}}{\Gamma(2-\alpha_q)}\sum\limits_{j=1}^{l}b_{j-1}\|\vartheta_0\|+ h^2L_f(1+L_{\chi})\frac{l(l+1)}{2}+ lh\|f(0,\chi(t_{-n}))\|.
	\end{eqnarray*}
	Using discrete version of Gronwall inequality \cite{AD2019}, we get
	\begin{eqnarray*}
		\|\vartheta_l-\chi(0)\|&\leq&\lambda \exp\bigg(\sum\limits_{i=1}^{l}\gamma_i\bigg)=C_2,
	\end{eqnarray*}
	where 
	\begin{eqnarray*}
		\lambda&=&lT\|A\chi(0)\|+\sum\limits_{q=1}^{k}\frac{a_qT^{1-\alpha_q}}{\Gamma(2-\alpha_q)}\sum\limits_{j=1}^{l-2}\left[(j-1)^{1-\alpha_q}+j^{1-\alpha_q}-1\right]\|\chi(0)\|\\
		&&+\sum\limits_{q=2}^{k}\frac{a_qT^{1-\alpha_q}}{\Gamma(2-\alpha_q)}\sum\limits_{j=2}^{l}b_{j-1}\|\vartheta_0\|+T^2L_f(1+L_{\chi})\frac{l(l+1)}{2}+ lT\|f(0,\chi(t_{-n}))\|,\\
		\gamma_1&=&1,~ \gamma_i=2\sum\limits_{q=1}^{k}\frac{a_qT^{1-\alpha_q}}{\Gamma(2-\alpha_q)}, \mbox{~for~} ~i=2,3,\ldots,l, \mbox{~and~} C_2=\lambda e^{\sum\limits_{i=1}^{l}\gamma_i}.
	\end{eqnarray*}
	We put $C=\max\{C_1,C_2\}$ and get
	\begin{eqnarray*}
		\max_{j=1,2,\dots,m}\left\|\vartheta_j-\chi(0)\right\|\leq C.
	\end{eqnarray*}
\end{proof}	
\begin{lemma}\label{3l2}
	If (A1)–(A4) hold, then for every $j = 1, 2,\dots,m$, there exists a constant $C'$,
	\begin{eqnarray}\label{3.8}
		\max_{j=1,2,\dots,m}\left\|\delta\vartheta_j\right\|\leq C'.
	\end{eqnarray}
\end{lemma}
\begin{proof}
	For $j=1$ in (\ref{3.6}), subtracting ${A}\vartheta_0$ from both the sides, we obtain
	\begin{eqnarray*}
		\delta\vartheta_1+\sum\limits_{q=1}^{k}\frac{a_qh^{1-\alpha_q}}{\Gamma(2-\alpha_q)}\delta\vartheta_1+{A}(\vartheta_1-\vartheta_0)=f(t_1,{\vartheta}_{1-n})-{A}\vartheta_0,
	\end{eqnarray*}
	or, we can write it as
	\begin{eqnarray*}
		\bigg(1+\sum\limits_{q=1}^{k}\frac{a_qh^{1-\alpha_q}}{\Gamma(2-\alpha_q)}\bigg)\delta\vartheta_1+{A}(\vartheta_1-\vartheta_0)=f(t_1,{\vartheta}_{1-n})-{A}\vartheta_0.
	\end{eqnarray*}
	Multiplying both sides by $\vartheta_1- \vartheta_0,$ we obtain
	\begin{eqnarray*}
		\bigg(1+\sum\limits_{q=1}^{k}\frac{a_qh^{1-\alpha_q}}{\Gamma(2-\alpha_q)}\bigg)\left\langle \delta\vartheta_1,\vartheta_1-\vartheta_0\right\rangle+\big\langle {A}(\vartheta_1-\vartheta_0),\vartheta_1-\vartheta_0\big\rangle \\\hspace{4cm}=\big\langle- A\vartheta_0,\vartheta_1-\vartheta_0\big\rangle +\big\langle f(t_1,{\vartheta}_{1-n}), \vartheta_1-\vartheta_0\big\rangle.
	\end{eqnarray*}
	Using m-accretivity of ${A}$, we have
	\begin{eqnarray*}
		\bigg(1+\sum\limits_{q=1}^{k}\frac{a_qh^{1-\alpha_q}}{\Gamma(2-\alpha_q)}\bigg)\left\|\delta\vartheta_1\right\|&\leq&\big\|A\vartheta_0\big\| +\big\|f(t_1,{\vartheta}_{1-n})\big\|.
	\end{eqnarray*}
	Since $\bigg(1+\sum\limits_{q=1}^{k}\dfrac{a_qh^{1-\alpha_q}}{\Gamma(2-\alpha_q)}\bigg)>1$, we have
	\begin{eqnarray*}
		\left\|\delta\vartheta_1\right\|
		&\leq&\big\|A\vartheta_0\big\| +\big\|f(t_1,{\vartheta}_{1-n})\big\|\\
		&\leq&\big\|A\vartheta_0\big\|+\big\|f(t_1,\chi(t_{1-n}))-f(0, \chi(t_{-n}))\big\|+\big\|f(0,\chi(t_{-n}))\big\|\\
		&\leq&C_3,
	\end{eqnarray*}
	where $C_3=\big\|A\vartheta_0\big\|+L_fT(1+L_{\chi})+\big\|f(0,\chi(t_{-n}))\big\|$.
	
	\noindent For $j\geq2$, we consider the $j$th equation in (\ref{3.6}) and the $(j-1)$th equation in (\ref{3.6}), and then subtract these two equalities, we get
	\begin{eqnarray*}
		\delta\vartheta_j-\delta\vartheta_{j-1}+{A}(\vartheta_j-\vartheta_{j-1})
		&=&\sum\limits_{q=1}^{k}\frac{a_qh^{1-\alpha_q}}{\Gamma(2-\alpha_q)}\sum\limits_{i=1}^{j-1}b_{j-i-1}\delta\vartheta_i-\sum\limits_{q=1}^{k}\frac{a_qh^{1-\alpha_q}}{\Gamma(2-\alpha_q)}\sum\limits_{i=1}^{j}b_{j-i}\delta\vartheta_i\\
		&&\quad+f(t_j,{\vartheta}_{j-n})-f(t_{j-1},{\vartheta}_{j-1-n}).
	\end{eqnarray*}
	We rewrite it as
	\begin{eqnarray*}
		&&\bigg(1+\sum\limits_{q=1}^{k}\dfrac{a_qh^{1-\alpha_q}}{\Gamma(2-\alpha_q)}\bigg)\delta\vartheta_j+{A}(\vartheta_j-\vartheta_{j-1})\\
		&&\quad=\left(1+\sum\limits_{q=1}^{k}\dfrac{a_qh^{1-\alpha_q}}{\Gamma(2-\alpha_q)}(1-b_1)\right)\delta\vartheta_{j-1}+\sum\limits_{q=1}^{k}\frac{a_qh^{1-\alpha_q}}{\Gamma(2-\alpha_q)}\sum\limits_{i=1}^{j-2}(b_{j-i-1}-b_{j-i})\delta\vartheta_{i}\\
		&&\quad\quad+f(t_j,{\vartheta}_{j-n})-f(t_{j-1},{\vartheta}_{j-1-n}).
	\end{eqnarray*}
	Multiplying both sides by $\vartheta_j-\vartheta_{j-1}$ and using the $m$-accretivity of $A$, we have	
	\begin{eqnarray*}
		&&\bigg(1+\sum\limits_{q=1}^{k}\dfrac{a_qh^{1-\alpha_q}}{\Gamma(2-\alpha_q)}\bigg)\left\|\delta\vartheta_j\right\|\\
		&&\quad\leq\left(1+\sum\limits_{q=1}^{k}\dfrac{a_qh^{1-\alpha_q}}{\Gamma(2-\alpha_q)}(1-b_1)\right)\left\|\delta\vartheta_{j-1}\right\|+\sum\limits_{q=1}^{k}\frac{a_qh^{1-\alpha_q}}{\Gamma(2-\alpha_q)}\sum\limits_{i=1}^{j-2}(b_{j-i-1}-b_{j-i})\left\|\delta\vartheta_{i}\right\|\\
		&&\quad\quad+L_f\left(|t_j-t_{j-1}|+\|\chi(t_{j-n})-\chi(t_{j-1-n})\|\right).
	\end{eqnarray*}
	Since $\bigg(1+\sum\limits_{q=1}^{k}\dfrac{a_qh^{1-\alpha_q}}{\Gamma(2-\alpha_q)}\bigg)>1$, we have
	\begin{eqnarray*}
		\left\|\delta\vartheta_j\right\|	&\leq&\left(1+\sum\limits_{q=1}^{k}\dfrac{a_qh^{1-\alpha_q}}{\Gamma(2-\alpha_q)}(1-b_1)\right)\left\|\delta\vartheta_{j-1}\right\|+\sum\limits_{q=1}^{k}\frac{a_qh^{1-\alpha_q}}{\Gamma(2-\alpha_q)}\sum\limits_{i=1}^{j-2}(b_{j-i-1}-b_{j-i})\left\|\delta\vartheta_{i}\right\|\\
		&&\quad+hL_f(1+L_{\chi}).
	\end{eqnarray*}
	Using discrete version of Gronwall inequality \cite{AD2019}, we get
	\begin{eqnarray*}
		\left\|\delta\vartheta_j\right\|\leq C_4.
	\end{eqnarray*}
	Now, we take $C'=\max\{C_3,C_4\}$, and hence, we get 
	\begin{eqnarray*}
		\max_{j=1,2,\dots,m}\left\|\delta\vartheta_j\right\|\leq C'.
	\end{eqnarray*}
	This completes the proof.
\end{proof}
\begin{lemma}\label{3l3} For every $n\in \mathbb{N},~j=1,2,\ldots,m,$ there exists a constant $\tilde{C}$ such that
	\begin{eqnarray}\label{3.9}
		\left\|{}^C{D}_t^{\alpha_q}\vartheta_j\right\|\leq \tilde{C}.
	\end{eqnarray}
\end{lemma}
\begin{proof}
	Since ${}^C{D}_t^{\alpha_q}\vartheta_j= \dfrac{h^{1-\alpha_q}}{\Gamma(2-\alpha_q)}\sum\limits_{i=1}^{j}b_{j-i}\delta\vartheta_i,$ using Lemma \ref{3l2} and  \cite[Lemma 1]{AD2019}, we have
	\begin{eqnarray*}
		\big\|{}^C{D}_t^{\alpha_q}\vartheta_j\big\|&\leq& \frac{C'C_{\alpha_q}}{\Gamma(2-\alpha_q)}\sum\limits_{i=1}^{j}\frac{h}{[(j-i+1)h]^{\alpha_q}}\\
		&\leq& \frac{C'C_{\alpha_q}}{\Gamma(2-{\alpha_q})}\sum\limits_{i=1}^{j}\int_{t_{i-1}}^{t_i}\frac{ds}{(t_j-s)^{\alpha_q}}\\
		&=& \frac{C'C_{\alpha_q}}{\Gamma(2-{\alpha_q})}\int_{0}^{t_j}\frac{ds}{(t_j-s)^{\alpha_q}}\\
		&\le&\frac{C'C_{\alpha_q}T^{1-{\alpha_q}}}{(1-{\alpha_q})\Gamma(2-{\alpha_q})}=\tilde{C}.
	\end{eqnarray*}
\end{proof}
We define the Rothe's approximation $\{{U}_n\}$ and corresponding step functions $\{X_n\} \subseteq C_{T_0}$ of polygonal functions
\begin{eqnarray} \label{3.4}
	{U}_n(t)=	\left \{ \begin{array}{lll} \chi(t), && t\in[-\nu,0],
		\\\vartheta_{j-1}+(t-t_{j-1})\delta\vartheta_j,&&t \in(t_{j-1},t_{j}],\quad j=1,2,\ldots,m,
	\end{array}\right.
\end{eqnarray}
\begin{eqnarray} \label{3.10}
	X_n(t)=\left \{ \begin{array}{ll}
		\chi(t), \quad &t\in [-\nu, 0],
		\\ \vartheta_j, \quad &t\in(t_{j-1},t_j], \quad	j=1,2,\ldots,m,
	\end{array}\right.\hspace{2.9cm}
\end{eqnarray}
\section{\textbf{Main Result}}
\begin{theorem}\label{th1}
	If the assumptions (A1)-(A4) hold, then there exists a strong solution $\vartheta\in C([-\nu, T_0]; H)$ of (\ref{1.1}). Moreover $\vartheta$ is Lipschitz continuous function.
\end{theorem}
\begin{proof}
	Let $n, p\in\mathbb{N}$ and $h=\frac{T_0}{n}$. Denote $ v_n$ by the following equation:
	\begin{eqnarray}\label{3.12}
		v_n(t)=	\left \{ \begin{array}{ll}
			0, \quad &t=0,
			\\\dfrac{h_n^{1-\alpha_q}}{\Gamma(2-\alpha_q)}\sum\limits_{i=1}^{j}b_{j-i}\delta\vartheta_i^n, \quad &t\in(t_{j-1}^n,t_j^n],\quad j=1,2,\dots,m.	 	
		\end{array}\right.
	\end{eqnarray}
	From the definition of ${}^C{D}^{\alpha_q}_t$, we have for $t\in (t_{j-1},t_j]$
	\begin{eqnarray}\label{ch2partialcapt}
		{}^C{D}^{\alpha_q}_tU_n(t)&=& \dfrac{1}{\Gamma(1-\alpha_q)}\displaystyle\int_{0}^{t}
		\dfrac{\frac{d}{ds}U_n(s)}{(t-s)^{\alpha_q}}ds\nonumber\\
		&=&\dfrac{1}{\Gamma(1-\alpha_q)}\left(\sum\limits_{i=1}^{j}\displaystyle\int_{t_{i-1}}^{t_{i}}\dfrac{\delta \vartheta_i}{(t-s)^{\alpha_q}}ds+\displaystyle\int_{t_{j}}^{t}\dfrac{\delta \vartheta_j}{(t-s)^{\alpha_q}}ds\right)\nonumber\\
		&=& \dfrac{1}{\Gamma(2-\alpha_q)}\sum\limits_{i=1}^{j}\left[(t-t_{i-1})^{1-\alpha_q}-(t-t_i)^{1-\alpha_q}\right]\delta \vartheta_i\nonumber\\
		&&\quad + \delta \vartheta_j\frac{(t-t_{j})^{1-\alpha_q}}{\Gamma(2-\alpha_q)}.
	\end{eqnarray}
	Also, from equation (\ref{3.12}), we have
	\begin{eqnarray}\label{ch2parttildcap}
		v_n(t)&=& \dfrac{h^{1-\alpha_q}}{\Gamma(2-\alpha_q)}\sum\limits_{i=1}^{j}b_{j-i}\delta \vartheta_i\nonumber\\
		&=& \dfrac{1}{\Gamma(2-\alpha_q)}\sum\limits_{i=1}^{j}\left[((j-i+1)h)^{1-\alpha_q}-((j-i)h)^{1-\alpha_q}\right]\delta \vartheta_i\nonumber\\
		&=&\dfrac{1}{\Gamma(2-\alpha_q)}\sum\limits_{i=1}^{j}\left[(t_j-t_{i-1})^{1-\alpha_q}-(t_j-t_i)^{1-\alpha_q}\right]\delta \vartheta_i
	\end{eqnarray}
	From equations (\ref{ch2partialcapt}) and (\ref{ch2parttildcap}), we have
	\begin{eqnarray}\label{caputoconvergence}
		\left\|v_n(t)-{}^CD^{\alpha_q}_tU_n(t)\right\|&\leq& \dfrac{1}{\Gamma(2-\alpha_q)}\sum\limits_{i=1}^{j}\left\|\delta \vartheta_i\right\|\Big[ \big[(t-t_{i-1})^{1-\alpha_q}-(t-t_i)^{1-\alpha_q} \big]\nonumber\\
		&&- \big[(t_j-t_{i-1})^{1-\alpha_q}-(t_j-t_i)^{1-\alpha_q}\big]\Big]+\dfrac{1}{\Gamma(2-\alpha_q)}\left\|\delta \vartheta_j\right\|(t-t_{j})^{1-\alpha_q}\nonumber\\
		&\leq& \dfrac{C}{\Gamma(2-\alpha_q)}\Big[t_j^{1-\alpha_q}-2t^{1-\alpha_q}+ 2(t-t_{j})^{1-\alpha_q}\Big]\to 0, \quad n\to \infty,
	\end{eqnarray}
	for $j=1,2,\dots,m$.
	
	By the definitions of ${U}_n, X_n$ and $v_n$, it is clear that (\ref{3.1}) is equivalent to the following problem 
	\begin{eqnarray}\label{3.13}
		\frac{d}{dt}U_n(t) +\sum_{q=1}^{k}a_qv_n(t)+{A}X_n(t)=f_n(t) \text{ for all } t\in [0,T_0],
	\end{eqnarray}
	where  \begin{eqnarray*}
		f_n(t)= f(t_j, X_n(t-\nu)), \quad t\in (t_{j-1}, t_j], \quad j=1,2,\dots,m.
	\end{eqnarray*}
	Analogously, for $\bar{h}=\frac{T_0}{p}$, we have
	\begin{eqnarray}\label{3.14}
		\frac{d}{dt}U_p(t) +\sum_{q=1}^{k}a_qv_p(t)+{A}X_p(t)=f_p(t) \text{ for all } t\in [0,T_0].
	\end{eqnarray}
	Subtracting the equalities (\ref{3.13}) and (\ref{3.14}), and then multiply the result by $X_n(t)-X_p(t)$, we have  
	\begin{eqnarray*}
		&&\left\langle\frac{d}{dt}U_n(t)- \frac{d}{dt}U_p(t), X_n(t)-X_p(t) \right\rangle+ 	\left\langle AX_n(t)-AX_p(t), X_n(t)-X_p(t)\right\rangle\\
		&&\quad=-\left\langle\sum_{q=1}^{k}a_qv_n(t)-\sum_{q=1}^{k}a_q{}^C{D}_t^{\alpha_q}U_n(t),X_n(t)-X_p(t)\right\rangle\\
		&&\quad\quad-\left\langle \sum_{q=1}^{k}a_q{}^C{D}_t^{\alpha_q}U_n(t)-\sum_{q=1}^{k}a_q{}^C{D}_t^{\alpha_q}U_p(t),  X_n(t)-X_p(t)
		\right\rangle\\\quad
		&&\quad\quad+\left\langle\sum_{q=1}^{k}a_qv_p(t)-\sum_{q=1}^{k}a_q{}^C{D}_t^{\alpha_q}U_p(t),X_n(t)-X_p(t)\right\rangle+\left \langle f_n(t)-f_p(t),X_n(t)-X_p(t)\right\rangle.
	\end{eqnarray*}
	Thus, we deduce the following estimate
	\begin{eqnarray}\label{estimate}
		&&\left\langle\frac{d}{dt}U_n(t)- \frac{d}{dt}U_p(t), U_n(t)-U_p(t) \right\rangle+ \left\langle\frac{d}{dt}U_n(t)- \frac{d}{dt}U_p(t), X_n(t)-U_n(t) \right\rangle\nonumber\\
		&&+\left\langle\frac{d}{dt}U_n(t)- \frac{d}{dt}U_p(t), U_p(t)-X_p(t) \right\rangle+	\left\langle AX_n(t)-AX_p(t), X_n(t)-X_p(t)\right\rangle\nonumber\\
		&&\quad=-\left\langle\sum_{q=1}^{k}a_qv_n(t)-\sum_{q=1}^{k}a_q{}^C{D}_t^{\alpha_q}U_n(t),X_n(t)-X_p(t)\right\rangle\nonumber\\
		&&\quad\quad-\left\langle \sum_{q=1}^{k}a_q{}^C{D}_t^{\alpha_q}U_n(t)-\sum_{q=1}^{k}a_q{}^C{D}_t^{\alpha_q}U_p(t),  X_n(t)-X_p(t)
		\right\rangle\nonumber\\
		&&\quad\quad+\left\langle\sum_{q=1}^{k}a_qv_p(t)-\sum_{q=1}^{k}a_q{}^C{D}_t^{\alpha_q}U_p(t),X_n(t)-X_p(t)\right\rangle\nonumber\\
		&&\quad\quad+\left \langle f_n(t)-f_p(t),X_n(t)-X_p(t)\right\rangle.
	\end{eqnarray}
	According to Lemma \ref{3l2}, we easily calculate that
	\begin{eqnarray}\label{boundofunminusxn}
		\left\|U_n(t)-X_n(t)\right\|= |t-t_j|\cdot\left\|\delta\vartheta_j\right\|\leq \frac{C}{n},
	\end{eqnarray}
	for all $t\in (t_{j-1}, t_j]$. From the Lemmas \ref{3l1}, \ref{3l2}, and inequality (\ref{boundofunminusxn}), we have the following inequality 
	\begin{eqnarray}
		\left \{ \begin{array}{lll} \left\|\dfrac{d}{dt}U_n(t)\right\|= \left\|\delta\vartheta_j\right\|\leq C,\vspace{0.2cm}	
			&\\\left\|X_n(t)-X_p(t)\right\|\leq \left\|X_n(t)-\chi(0)\right\|+\left\|X_p(t)-\chi(0)\right\|\leq C,
			\vspace{0.2cm}	
			&\\\left\|{}^C{D}_t^{\alpha_q}U_n(t)-{}^C{D}_t^{\alpha_q}X_n(t)\right\|\leq C,
			\vspace{0.2cm}	
			&\\\left\|{}^C{D}_t^{\alpha_q}U_n(t)-{}^C{D}_t^{\alpha_q}U_p(t)\right\|\leq C,
			\vspace{0.2cm}	
			&\\\left\|f_n(t)-f_p(t)\right\|\leq L_f\left[|t_j-t_{\bar{j}}+ \left\|X_n(t-\nu)-X_p(t-\nu)\right\|\right]\leq \epsilon_{np}(t),
		\end{array}\right.
	\end{eqnarray}
	where $\epsilon_{np}(t)=L_f\left[|t_j-t_{\bar{j}}+ \left\|X_n(t-\nu)-U_n(t)\right\| + \left\|X_p(t-\nu)-U_n(t)\right\|\right]$
	for all $t\in (t_{j-1}, t_j]$  and  $t\in (t_{\bar{j}-1}, t_{\bar{j}}], 1\leq j\leq n, 1\leq \bar{j}\leq p$. Therefore, $\epsilon_{np}(t)\to 0$ as $n,p \to \infty$ uniformly on $[0, T_0]$. 
	
	By using $m$-accretivity of the operator $A$, Lemmas \ref{3l1}, \ref{3l2}, the inequality (\ref{caputoconvergence}), from (\ref{estimate}), we get
	\begin{eqnarray}\label{estimate1}
		&&\frac{1}{2}\frac{d}{dt}\left\|U_n(t)- U_p(t)\right\|^2\nonumber\\
		&&\leq C\left[\left\|U_n(t)-X_n(t)\right\|+\left\|U_p(t)-X_p(t)\right\|+  \left\|f_n(t)-f_p(t)\right\|+ \sum_{q=1}^{k}\left[\frac{1}{n^{1-\alpha_q}}+\frac{1}{p^{1-\alpha_q}}\right]\right]\nonumber\\
		&&\quad+\sum_{q=1}^{k}a_q\left\|{}^C{D}_t^{\alpha_q}U_n(t)-{}^C{D}_t^{\alpha_q}U_p(t)\right\|\left\|X_n(t)-X_p(t)\right\|\nonumber\\
		&&\leq C\left[\frac{1}{n}+\frac{1}{p}+ \epsilon_{np}(t)+ \sum_{q=1}^{k}\left[\frac{1}{n^{1-\alpha_q}}+\frac{1}{p^{1-\alpha_q}}\right]\right]\nonumber\\
		&&\quad+\sum_{q=1}^{k}a_q\left\|{}^C{D}_t^{\alpha_q}U_n(t)-{}^C{D}_t^{\alpha_q}U_p(t)\right\|\left\|X_n(t)-U_n(t)\right\|	\nonumber\\
		&&\quad+\sum_{q=1}^{k}a_q\left\|{}^C{D}_t^{\alpha_q}U_n(t)-{}^C{D}_t^{\alpha_q}U_p(t)\right\|\left\|X_p(t)-U_p(t)\right\|	\nonumber\\
		&&\quad+\sum_{q=1}^{k}a_q\left\|{}^C{D}_t^{\alpha_q}U_n(t)-{}^C{D}_t^{\alpha_q}U_p(t)\right\|\left\|U_n(t)-U_p(t)\right\|\nonumber\\
		&&\leq C\left[\frac{1}{n}+\frac{1}{p}+ \epsilon_{np}(t)+ \sum_{q=1}^{k}\left[\frac{1}{n^{1-\alpha_q}}+\frac{1}{p^{1-\alpha_q}}\right]\right]\nonumber\\
		&&\quad+\sum_{q=1}^{k}\dfrac{a_q}{\Gamma(1-\alpha_q)}\int_{0}^{t}(t-s)^{\alpha_q}\frac{d}{ds}\left\|U_n(s)-U_p(s)\right\|\left\|U_n(t)-U_p(t)\right\|ds\nonumber\\
		&&\leq 	C\left[\frac{1}{n}+\frac{1}{p}+ \epsilon_{np}(t)+ \sum_{q=1}^{k}\left[\frac{1}{n^{1-\alpha_q}}+\frac{1}{p^{1-\alpha_q}}\right]\right]\nonumber\\
		&&\quad+\sum_{q=1}^{k}\dfrac{a_q\alpha_q}{\Gamma(1-\alpha_q)}\int_{0}^{t}(t-s)^{\alpha_q-1}\left\|U_n(s)-U_p(s)\right\|\left\|U_n(t)-U_p(t)\right\|ds.
	\end{eqnarray}
	We consider the last term in the inequality (\ref{estimate1}). It follows from the Cauchy inequality with $\epsilon>0$, see \cite{E1990}, that
	\begin{eqnarray}\label{integralestimate}
		&&\sum_{q=1}^{k}\dfrac{a_q\alpha_q}{\Gamma(1-\alpha_q)}\int_{0}^{t}(t-s)^{\alpha_q-1}\left\|U_n(s)-U_p(s)\right\|\left\|U_n(t)-U_p(t)\right\|ds\nonumber\\
		&&\leq\epsilon\sum_{q=1}^{k}\dfrac{a_q\alpha_q}{\Gamma(1-\alpha_q)}\int_{0}^{t}(t-s)^{\alpha_q-1}\left\|U_n(s)-U_p(s)\right\|^2ds\nonumber\\
		&&\quad + \sum_{q=1}^{k}\dfrac{a_qT_0^{\alpha_q}}{4\epsilon\Gamma(1-\alpha_q)}\left\|U_n(t)-U_p(t)\right\|^2\nonumber\\
		&&\leq\sum_{q=1}^{k}\dfrac{a_qT_0^{\alpha_q}}{4\epsilon\Gamma(1-\alpha_q)}\left\|U_n(t)-U_p(t)\right\|^2 + \epsilon\sum_{q=1}^{k}\dfrac{a_qT_0^{\alpha_q}}{\Gamma(1-\alpha_q)}\sup_{t\in [0, T_0]}\left\|U_n(t)-U_p(t)\right\|^2.
	\end{eqnarray}  
	Using the inequality (\ref{integralestimate}), we can write the equation (\ref{estimate1}) as
	\begin{eqnarray}\label{estimate2}
		\frac{1}{2}\frac{d}{dt}\left\|U_n(t)-U_p(t)\right\|^2	&\leq &	C\left[\frac{1}{n}+\frac{1}{p}+ \epsilon_{np}(t)+ \sum_{q=1}^{k}\left[\frac{1}{n^{1-\alpha_q}}+\frac{1}{p^{1-\alpha_q}}\right]\right]\nonumber\\
		&&\quad\quad+\sum_{q=1}^{k}\dfrac{a_qT_0^{\alpha_q}}{4\epsilon\Gamma(1-\alpha_q)}\left\|U_n(t)-U_p(t)\right\|^2\nonumber\\
		&&\quad\quad+ \epsilon\sum_{q=1}^{k}\dfrac{a_qT_0^{\alpha_q}}{\Gamma(1-\alpha_q)}\sup_{t\in [0, T_0]}\left\|U_n(t)-U_p(t)\right\|^2.
	\end{eqnarray}
	Integrating the inequality (\ref{estimate2}) over $[0, T_0]$, we obtain that
	\begin{eqnarray}\label{estimate3}
		\frac{1}{2}\left\|U_n(t)-U_p(t)\right\|^2	&\leq &	C\left[\frac{1}{n}+\frac{1}{p}+ \epsilon_{np}(t)+ \sum_{q=1}^{k}\left[\frac{1}{n^{1-\alpha_q}}+\frac{1}{p^{1-\alpha_q}}\right]\right]\nonumber\\
		&&\quad\quad+ \epsilon\sum_{q=1}^{k}\dfrac{a_qT_0^{\alpha_q+1}}{\Gamma(1-\alpha_q)}\sup_{t\in [0, T_0]}\left\|U_n(t)-U_p(t)\right\|^2\nonumber\\
		&&\quad\quad+\sum_{q=1}^{k}\dfrac{a_qT_0^{\alpha_q}}{4\epsilon\Gamma(1-\alpha_q)}\int_{0}^{t}\left\|U_n(s)-U_p(s)\right\|^2ds.
	\end{eqnarray}
Choose $\epsilon=\frac{1}{4}\left[\sum\limits_{q=1}^{k}\dfrac{a_qT_0^{\alpha_q+1}}{\Gamma(1-\alpha_q)}\right]^{-1}$, so that $\frac{1}{2}-\epsilon\left(\sum_{q=1}^{k}\dfrac{a_qT_0^{\alpha_q+1}}{\Gamma(1-\alpha_q)}\right)=\frac{1}{4}$. This leads to 
\begin{eqnarray}\label{estimate4}
		\frac{1}{2}\left\|U_n(t)-U_p(t)\right\|^2	&\leq & 	\frac{1}{4}\sup_{t\in [0, T_0]}\left\|U_n(t)-U_p(t)\right\|^2\nonumber\\
		&\leq&	C\left[\frac{1}{n}+\frac{1}{p}+ \epsilon_{np}(t)+ \sum_{q=1}^{k}\left[\frac{1}{n^{1-\alpha_q}}+\frac{1}{p^{1-\alpha_q}}\right]\right]\nonumber\\
		&&\quad\quad+\sum_{q=1}^{k}\dfrac{a_qT_0^{\alpha_q}}{4\epsilon\Gamma(1-\alpha_q)}\int_{0}^{t}\left\|U_n(s)-U_p(s)\right\|^2ds.
	\end{eqnarray}
	Thus, applying the integral version of Gronwall inequality, see \cite[Lemma 2.31. p.49]{MWM2006}, we conclude that
	\begin{eqnarray*}
		\left\|U_n(t)-U_p(t)\right\|^2	&\leq &C\left[\frac{1}{n}+\frac{1}{p}+ \epsilon_{np}(t)+ \sum_{q=1}^{k}\left[\frac{1}{n^{1-\alpha_q}}+\frac{1}{p^{1-\alpha_q}}\right]\right],
	\end{eqnarray*}
	for all $t\in [0, T_0]$, where $C$ is independent of $j, \bar{j}$ and $t$. This shows that $U_n$ is a Cauchy sequence in $[0, T_0]$, and therefore, there exists $\vartheta\in C([-\nu, T_0];H)$, such that $U_n\to \vartheta$ as $n\to \infty$. Further, by the fact that the function  $U_n$ is uniformly Lipschitz continuous for all $n$, we conclude that $\vartheta$ is Lipschitz continuous as well. 
	
	Next, we will show that  $\vartheta$ is strong solution of (\ref{1.1}). From the convergence $U_n\to \vartheta$ in $C([-\nu, T_0]; H)$ as $n\to \infty$, we know that $U_n(t)\to \vartheta(t)$ in $H$ for all $t\in [0,T_0]$ as $n\to \infty$. However, the inequality (\ref{boundofunminusxn}) ensures that $X_n(t)\to \vartheta(t)$ in $H$ for all $t\in [0,T_0]$, as $n\to \infty$. According to the $m$-accretivity of the operator $A$, the
	uniform convexity of $H$, from Proposition \ref{accrativeproperty}, we know that $A$ is demiclosed, thus, $AX_n(t)\rightharpoonup A\vartheta(t)$ in $H$ for all $t\in [0,T_0]$, as $n\to \infty$.
	
	In addition, it is clear that ${}^C{D}_t^{\alpha_q}U_n(t)\to {}^C{D}_t^{\alpha_q}\vartheta(t)$ in $H$ for all $t\in [0,T_0]$, as $n\to \infty$. Then, directly from (\ref{caputoconvergence}), we also get $\sum\limits_{q=1}^{k}a_qv_n(t)\to\sum\limits_{q=1}^{k}a_q{}^C{D}_t^{\alpha_q}\vartheta(t)$ in $H$ for all $t\in [0,T_0]$, as $n\to \infty$. Moreover, we easily obtain $f_n(t)\to f(t,\vartheta(t-\nu))$ in $C([-\nu, T_0]; H)$ as $n\to \infty$.
	
	Let $z^*\in H$. We multiply (\ref{3.13}) by $z^*$, and integrate the result on $[0, T_0]$, we have		
	\begin{eqnarray*}
		\left\langle U_n(t)- U_0,z^*\right\rangle +  \displaystyle\int_{0}^{t}\sum_{q=1}^{k}a_q\left\langle v_n(s),z^*\right\rangle ds+\displaystyle\int_{0}^{t}\left\langle {A}X_n(s), z^*\right\rangle ds= \displaystyle\int_{0}^{t}\left\langle f_n(s),z^*\right\rangle ds.
	\end{eqnarray*}
	Letting $n\to \infty$ and applying Lebesgue dominated convergence theorem, we infer that
	\begin{eqnarray*}
		\left\langle \vartheta(t)- U_0,z^*\right\rangle +  \displaystyle\int_{0}^{t}\sum_{q=1}^{k}a_q\left\langle {}^C{D}_t^{\alpha_q}  \vartheta(s),z^*\right\rangle ds+\displaystyle\int_{0}^{t}\left\langle {A}\vartheta(s), z^*\right\rangle ds= \displaystyle\int_{0}^{t}\left\langle f(s, \vartheta(s-\nu)),z^*\right\rangle ds,
	\end{eqnarray*}
	for all $z^*\in H$. Since $z^*$ is arbitrary, we can easily get 
	\begin{eqnarray*}
		\vartheta(t)- U_0+ \sum_{q=1}^{k}a_q \displaystyle\int_{0}^{t}{}^C{D}_t^{\alpha_q} \vartheta(s)ds+\displaystyle\int_{0}^{t}{A}\vartheta(s)ds= \displaystyle\int_{0}^{t}f(s, \vartheta(s-\nu))ds.
	\end{eqnarray*}
	Hence, we obtain that
	\begin{eqnarray*}
		\dfrac{d}{dt}\vartheta(t)+\sum\limits_{q=1}^{k}a_q{}^C{D}_t^{\alpha_q}\vartheta(t)+{A}\vartheta(t)=f(t,\vartheta(t-\nu)),
	\end{eqnarray*} 
	for a.e. $t\in [0,T_0]$. From the convergence $U_n(0)\to \vartheta(0)$ in ${H}$ as $n\to \infty$, we have $\vartheta(0)= \chi(0)$. This means that $\vartheta$ is a strong solution of (\ref{1.1}) in $[0,T_0]$.
\end{proof}

\section{\textbf{Applications}}
\begin{example}
	Consider the following differential equation:
	\begin{eqnarray}  \label{5.1}
		\left \{ \begin{array}{lll} \dfrac{\partial \vartheta(t,x)}{\partial t}+\sum\limits_{q=1}^{k}a_q{}^C{D}_t^{\alpha_q}\vartheta(t,x)-\dfrac{\partial^2\vartheta(t,x)}{\partial x^2}=f(t,\vartheta(t-2\pi,x)), 
			&\\\hspace{6.3cm}t\in [0,2\pi],~x\in[0,1],
			&\\ \vartheta(t,x)=\phi(t,x), \quad t \in [-2\pi,0],~x\in[0,1],
			&\\\vartheta(t,0)=\vartheta(t,1)=0, \quad t \in [0, 2\pi],
		\end{array}\right.
	\end{eqnarray}
	where $a_q\geq 0$,~${}^C{D}_t^{\alpha_q}$ represents Caputo derivative of order $0<\alpha_q<1$ for $q=1,2,\dots,k$ and $\vartheta:[-2\pi,2\pi]\times[0,1]\to \mathbb{R}$ is an unknown function. The function $f$ is a Lipschitz continuous function in $t\times \vartheta$ and $\phi(t,x): [-2\pi,0]\times [0,\pi]\to \mathbb{R}$ is a given function.
	
	\noindent Let us take ${H}=L^2[0,1]$ and define an operator ${A}$ by ${A}\vartheta=-\vartheta''$ with $$D({A})=\{\vartheta \in {H} : \vartheta, \vartheta'~ \mbox{are absolutely continuous},~ \vartheta'' \in {H}, \mbox{ and }~\vartheta(0)=\vartheta(1)=0\}.$$ 
	Clearly, $-{A}$ is the generator of a compact analytic semigroup of contractions. Therefore, by \cite[Theorem 2.3.3]{A1983},  $(I+{A})^{-1}$ is compact. Next, we use $\vartheta: [-2\pi,2\pi]\to {H},~\phi:[-2\pi,0]\to {H}$ and $f: [0,2\pi]\times C([-2\pi,0];{H})\to {H}$ to represent $\vartheta(t,x)$,~$\phi(t,x)$ and $f(t,\vartheta(t-2\pi,x))$, respectively. Then $(\ref{5.1})$ can be written as in abstract form of (\ref{1.1}). Further, it is easy to check that (A1)-(A4) are satisfied. Therefore, Theorem~\ref{th1} guarantees  that there exists a strong solution of (\ref{5.1}).
\end{example}
\section{\textbf{Conclusions}}
The aim of the present work is to establish the existence result of a strong solution of the multi-term fractional differential equations with delay. To arrive at our conclusion, we have used the approach of Rothe's method. Finally, we have given an example to illustrate the finding. We will focus on the existence and uniqueness of solutions for a quasi-linear fractional differential equation.

\end{document}